\documentclass{amsart}

\usepackage{latexsym}
\usepackage{amsmath}
\usepackage{amsthm}
\usepackage{amssymb}
\usepackage{amsfonts}
\usepackage{fancyhdr}
\usepackage{mathrsfs}                           

\theoremstyle{definition}
\newtheorem{thm}{Theorem}[section]

\newtheorem{prop}[thm]{Proposition}
\newtheorem{cor}[thm]{Corollary}
\newtheorem{lemma}[thm]{Lemma}

\newtheorem{definition}{Definition}

\newtheorem*{remark}{Remark}

\numberwithin{equation}{section}

\newcommand{\R}{\mathbb{R}}

\newcommand{\N}{\mathbb{N}}
\newcommand{\p}{\partial}

\newcommand{\Det}[1]{\mbox{Det}\left(#1\right)}
\newcommand{\Detg}[2]{\mbox{Det}_{#2}\left(#1\right)}
\newcommand{\tr}[1]{~\mbox{tr}\left(#1\right)}
\newcommand{\W}{W^{1,n}_{\textit{loc}}}
\renewcommand{\l}{\left}
\renewcommand{\r}{\right}

\renewcommand{\det}[1]{~\mbox{det}\left(#1\right)}
\newcommand{\eps}{\varepsilon}
\renewcommand{\d}{\partial}
\renewcommand{\d}{\delta}

\newcommand{\mR}{\mathbb{R}}

\newcommand{\supp}{\mathop{\rm supp}}
\newcommand{\abs}[1]{\lvert #1 \rvert}
\newcommand{\norm}[1]{\lVert #1 \rVert}

\renewcommand{\a}{\alpha}
\newcommand{\e}{\epsilon}

\newcommand{\mA}{\mathcal{A}}
\newcommand{\mdiv}{\mathrm{div}}
\renewcommand{\dh}{\Delta_h}
\newcommand{\vp}{\varphi}

\title[$n$-harmonic coordinates and regularity of conformal maps]{$n$-harmonic coordinates and the regularity of conformal mappings}

\author{Tony Liimatainen}
\address{Department of Mathematics and Statistics \\ University of Jyv\"askyl\"a \\ P.O.Box 35, FI-40014 University of Jyv\"askyl\"a, Finland}
\email{tony.t.liimatainen@jyu.fi}

\author{Mikko Salo}
\address{Department of Mathematics and Statistics \\ University of Jyv\"askyl\"a \\ P.O.Box 35, FI-40014 University of Jyv\"askyl\"a, Finland}
\email{mikko.j.salo@jyu.fi}

\begin{document}

\begin{abstract}
This article studies the smoothness of conformal mappings between two Riemannian manifolds whose metric tensors have limited regularity. We show that any bi-Lipschitz conformal mapping or $1$-quasiregular mapping between two manifolds with $C^r$ metric tensors ($r > 1$) is a $C^{r+1}$ conformal (local) diffeomorphism. This result was proved in \cite{Iwaniec_thesis, Reshetnyak78, Shefel}, but we give a new proof of this fact. The proof is based on $n$-harmonic coordinates, a generalization of the standard harmonic coordinates that is particularly suited to studying conformal mappings. We establish the existence of a $p$-harmonic coordinate system for $1 < p < \infty$ on any Riemannian manifold. 
\end{abstract}

\maketitle

\section{Introduction}

This article addresses the following question: given a conformal mapping between two smooth ($=C^{\infty}$) manifolds having metric tensors of limited regularity, how regular is the mapping? This question is very classical if one considers isometries instead of conformal mappings~\cite{Myers, HW,  Calabi}: it is known that a distance preserving homeomorphism between two Riemannian manifolds with $C^r$ ($r > 0$) metric tensors is in fact a $C^{r+1}$ isometry. We refer to~\cite{Taylor} for a recent treatment of this topic and related questions based on the systematic use of harmonic coordinates.

There has been considerable interest in the regularity of conformal mappings between domains in Euclidean space. Classical proofs of the Liouville theorem show that any $C^3$ conformal mapping between domains in $\mR^n$, $n \geq 3$, is a restriction of a M\"obius transformation and therefore $C^{\infty}$. The regularity assumption has been reduced in several works. It is known that any $W^{1,n}_{loc}$ solution of the Cauchy-Riemann system is a M\"obius transformation (see~\cite{Iwaniec}), and in even dimensions $n \geq 4$ this has been improved to the optimal case of $W^{1,n/2}_{loc}$ solutions~\cite{Iwaniec}. 

For conformal mappings between general Riemannian manifolds, it was proved by Lelong-Ferrand~\cite{Ferrand} that a conformal homeomorphism between two $C^{\infty}$ Riemannian manifolds is $C^{\infty}$. The works of Reshetnyak\ \cite{Reshetnyak78}, Iwaniec\ \cite{Iwaniec_thesis} and Shefel\ \cite{Shefel} imply the natural regularity for conformal mappings: a conformal mapping between two Riemannian manifolds with $C^r$ ($r>0$ not an integer) metric tensors is $C^{r+1}$ regular. We also mention the interesting recent articles \cite{Capogna} and \cite{MatveevTroyanov} which establish smoothness for subRiemannian and Finsler isometries.

In this paper we give a new proof that any bi-Lipschitz conformal mapping, or more generally any $1$-quasiconformal mapping with $W^{1,n}_{loc}$ regularity, between two Riemannian manifolds with $C^r$ ($r > 1$) metric tensors is in fact a $C^{r+1}$ conformal diffeomorphism. Similarly as in~\cite{Taylor}, this is done via a special coordinate system in which tensors on the manifold have maximal regularity. However, instead of harmonic coordinates that are useful in the regularity analysis of isometries, we employ a system of \emph{$n$-harmonic coordinates}. The existence and regularity of such coordinate systems on any Riemannian manifold is the other main result in this paper.

To explain $n$-harmonic coordinates, we recall that a function $u$ on a Riemannian manifold is called $p$-harmonic ($1 < p < \infty$) if it satisfies the nonlinear elliptic equation ($p$-Laplace equation) 
$$
\delta(\abs{du}^{p-2} du) = 0.
$$
Here $\abs{\,\cdot\,}$ is the norm induced by the Riemannian metric, $d$ is the exterior derivative, and $\delta$ is the codifferential (the adjoint of $d$ in the $L^2$ inner product on differential forms). A coordinate system is called $p$-harmonic if each coordinate function is $p$-harmonic.

The special property of $n$-harmonic functions is that pullbacks by conformal mappings preserve this class. In this way, any conformal mapping can be expressed as the composition of an $n$-harmonic coordinate chart and the inverse of another such chart. Any system of $n$-harmonic coordinates on a manifold with $C^r$ metric tensor ($r > 1$) always has $C^{r+1}_*$ regularity, and the regularity of conformal mappings follows immediately from these facts.

Any set of $n$-harmonic coordinates depends only on the conformal class of the Riemannian metric. In~\cite{LiimatainenSalo1} we used this fact to show that in $n$-harmonic coordinates the usual conformal curvature tensors may be viewed as elliptic operators. This allowed us to prove elliptic regularity results for solutions of conformal curvature equations and to characterize local conformal flatness in low regularity settings. The $n$-harmonic coordinates seem like a natural tool, and it is an interesting question whether they can be used also for other questions in conformal geometry.


This article may be considered as a partial complement to~\cite{Taylor}, and the formulation and proof of Theorem \ref{mainthm} concerning regularity of conformal mappings are very close to the analogous theorem in~\cite{Taylor} for the case of isometries. The existence of $p$-harmonic coordinates, or more generally $\mA$-harmonic coordinates, for any $1 < p < \infty$ is proved in a similar way as the existence of harmonic coordinates (the special case $p=2$) in~\cite{TaylorToolsForPDE} by perturbing the Cartesian coordinates and solving suitable Dirichlet problems in small balls. However, the underlying equation is nonlinear and solutions have a priori only $C^{1+\alpha}$ regularity for some $\alpha > 0$, so one needs to work with the weak form of the equation. A similar argument was given in~\cite{Salo} where solutions to a variable coefficient $p$-Laplace type equation were obtained by perturbing special $p$-harmonic functions. The present work was motivated by \cite{LS}, where regularity of conformal mappings was a key point. Finally we mention that the results in this paper have been extended to $C^r$ coefficients, $r > 0$, in \cite{JulinLiimatainenSalo}.

\subsection*{Notation}

The space $C^r$, when $r$ is a nonnegative integer, denotes the space of $r$ times continuously differentiable functions. If $r > 0$ is not an integer, so $r = l + \alpha$ for some integer $l$ where $0 < \alpha < 1$, then $C^r$ is the H\"older space consisting of all $C^l$ functions whose $l$th derivatives are $\alpha$-H\"older continuous. We will also use the Zygmund spaces $C^r_*$, defined for $r > 0$ to be the set of those functions $u \in L^{\infty}(\mR^n)$ such that 
$$
\norm{u}_{C^r_*} := \sup_{k \geq 0} 2^{kr} \norm{\psi_k(D) u}_{L^{\infty}} < \infty
$$
where $\psi_k(D)$ is a Littlewood-Paley partition of unity. For more details see~\cite{Taylor3}. By $W^{k,p}$ we denote the Sobolev space of functions whose weak partial derivatives up to order $k$ are in $L^p$. These spaces are well defined also on smooth manifolds.

The matrix of a Riemannian metric $g$ in local coordinates is $(g_{jk})$, its inverse matrix is $(g^{jk})$, and its determinant is $\abs{g}$. We use the Einstein summation convention where repeated indices in lower and upper position are summed.

\section{$p$-harmonic coordinates}

If $(M,g)$ is a Riemannian manifold, then the $p$-Laplace operator is given in local coordinates by 
\begin{equation}\label{p_harm_eq}
\delta(\abs{du}^{p-2} du) = -\abs{g}^{-1/2} \partial_j (\abs{g}^{1/2} g^{jk} (g^{ab} \partial_a u \partial_b u)^{\frac{p-2}{2}} \partial_k u).
\end{equation}
The main theorem of this section asserts that whenever $(M,g)$ is a Riemannian manifold with $C^r$ metric tensor and $1 < p < \infty$, then near any point of $M$ there exist local coordinates all of whose coordinate functions are $p$-harmonic.

\begin{thm}[$p$-harmonic coordinates]\label{p-harm-coord}
Let $(M,g)$ be a Riemannian manifold whose metric is of class $C^r$, $r>1$, in a local coordinate chart about some point $x_0 \in M$. Let also $1 < p < \infty$. There exists a local coordinate chart near $x_0$ whose coordinate functions are $p$-harmonic and have $C^{r+1}_*$ regularity, and given any $\eps > 0$, one can arrange so that metric satisfies $\abs{g_{jk}(x_0) - \delta_{jk}} <  \eps$ for $j,k=1,\ldots,n$. Moreover, all $p$-harmonic coordinates near $x_0$ have $C^{r+1}_*$ regularity.
\end{thm}

\begin{remark}
If $\Gamma_{ij}^k$ are the Christoffel symbols of the metric and $\Gamma^k=g^{ij}\Gamma_{ij}^k$, it is well known that the coordinate function $x^k$ is $2$-harmonic if and only if $\Gamma^k = 0$. More generally, we have that $x^k$ is $p$-harmonic if and only if
$$
\Gamma^k=\frac{p-2}{2}g^{ki}\p_i\log{g^{kk}}=(2-p)\frac{g^{ki}\Gamma_{ij}^k g^{jk}}{g^{kk}}.
$$
This follows by a straightforward calculation (cf.~\cite{Kazdan_Warner}) from the $p$-harmonic equation~\eqref{p_harm_eq} and the identity $\Gamma^k=-|g|^{-1/2}\p_i(|g|^{1/2}g^{ik})$.
\end{remark}

Since this theorem is a local statement, it is sufficient to consider a corresponding result in $\mR^n$. Locally, $p$-harmonic functions in $(M,g)$ are solutions of the $\mA$-harmonic equation $\mdiv \,\mA(x,\nabla u) = 0$ in a subset of $\mR^n$ where 
\begin{equation} \label{A_def}
\mA^j(x,q) = \abs{g(x)}^{1/2} g^{jk}(x) (g^{ab}(x) q_a q_b)^{\frac{p-2}{2}} q_k.
\end{equation}
In general, $\mA$-harmonic equations were considered in \cite{LU} as a class of quasilinear equations that includes the Euler-Lagrange equations for many variational problems. These equations appear in the theory of quasiregular maps \cite{HKM, Iwaniec}, and in fact the pullback of an $\mA$-harmonic function by a quasiregular map solves another equation of the same type \cite{HKM}. The main example is \eqref{A_def} related to some matrix $g$, but often this matrix is only measurable and it is technically convenient to consider general $\mA$ satisfying certain conditions (see for instance \cite{OP}).

We recall a few standard facts on the $\mA$-harmonic equation. We will later consider a family of equations $\mdiv \,\mA_{\eps}(x,\nabla u) = 0$ where $\mA_{\eps}$ satisfy uniform bounds with respect to $0 \leq \eps \leq 1$, and therefore we need to pay attention to the constants in norm estimates. The first fact concerns norm estimates for solutions of the Dirichlet problem.

\begin{prop} \label{Aharmonic_dirichlet}
Let $\Omega$ be a bounded open set in $\mR^n$ and let $1 < p < \infty$. Consider a function $\mA: \Omega \times \mR^n \to \mR^n$ such that for some $\alpha, \beta > 0$: 
\begin{gather*}
x \mapsto \mA(x,\xi) \text{ is measurable for all } \xi \in \mR^n, \label{aharm1} \\
\xi \mapsto \mA(x,\xi) \text{ is continuous for a.e. } x \in \mR^n, \label{aharm2} \\
\mA(x,\xi) \cdot \xi \geq \alpha \abs{\xi}^p \text{ for a.e.~$x \in \Omega$ and for all $\xi \in \mR^n$}, \label{aharm3} \\
\abs{\mA(x,\xi)} \leq \beta \abs{\xi}^{p-1} \text{ for a.e.~$x \in \Omega$ and for all $\xi \in \mR^n$}, \label{aharm4} \\
(\mA(x,\xi)-\mA(x,\zeta)) \cdot (\xi-\zeta) > 0 \ \ \text{for a.e.~$x \in \Omega$ and for $\xi, \zeta \in \mR^n, \xi \neq \zeta$.}
\end{gather*}
Given $f \in W^{1,p}(\Omega)$, there is $u \in W^{1,p}(\Omega)$ such that $\mdiv \mA(x,\nabla u) = 0$ in $\Omega$ with $u-f \in W^{1,p}_0(\Omega)$. Moreover, 
$$
\norm{u}_{W^{1,p}(\Omega)} \leq C \norm{f}_{W^{1,p}(\Omega)}
$$
where $C$ only depends on $\alpha$, $\beta$, $p$, and $\Omega$.
\end{prop}
\begin{proof}
The existence of $u$ for a given $f$ follows from a monotonicity argument \cite[Appendix I]{HKM}. For norm estimates, note that $u-f \in W^{1,p}_0(\Omega)$ may be used as a test function which implies 
$$
\int_{\Omega} \mA(x,\nabla u) \cdot \nabla(u-f) \,dx = 0.
$$
Using the assumptions on $\mA$ and H\"older's inequality we obtain 
\begin{align*}
\alpha \int_{\Omega} \abs{\nabla u}^p \,dx &\leq \int_{\Omega} \mA(x,\nabla u) \cdot \nabla u \,dx = \int_{\Omega} \mA(x,\nabla u) \cdot \nabla f \,dx \\
 &\leq \int_{\Omega} \beta \abs{\nabla u}^{p-1} \abs{\nabla f} \,dx \leq \beta \norm{\nabla u}_{L^p(\Omega)}^{p-1} \norm{\nabla f}_{L^p(\Omega)}.
\end{align*}
It follows that 
$$
\norm{\nabla u}_{L^p(\Omega)} \leq \frac{\beta}{\alpha} \norm{\nabla f}_{L^p(\Omega)}.
$$
Also, by the Poincar\'e inequality 
$$
\norm{u-f}_{L^p(\Omega)} \leq C_{p,\Omega} \norm{\nabla u - \nabla f}_{L^p(\Omega)} \leq C \norm{\nabla f}_{L^p(\Omega)}.
$$
The estimate $\norm{u}_{W^{1,p}} \leq C \norm{f}_{W^{1,p}(\Omega)}$ follows.
\end{proof}


The next result concerns higher regularity of $\mA$-harmonic functions, and requires further assumptions on $\mA$. We recall that even in Euclidean space, solutions of the $p$-Laplace equation always have $C^{1+\alpha}$ regularity for some $\alpha > 0$ but are not $C^2$ in general. However, if one knows a priori that the solution has nonvanishing gradient, then the solution is as regular as the function $\mA$ naturally allows. 
We write $B_r = B(0,r)$ for the open ball of radius $r > 0$ centered at the origin in $\mR^n$, and remark that the Zygmund space $C^r_*$ coincides with the H\"older space $C^r$ if $r$ is not an integer~\cite{Taylor3}.

\begin{prop} \label{Aharmonic_regularity}
Let $1 < p < \infty$ and let $\mA: B_1 \times (\mR^n \setminus \{0\}) \to\mR^n$ be a $C^1$ function such that for some $\delta, \beta > 0$, one has for $x \in B_1$ and $\xi, \zeta \in \mR^n$:
\begin{gather}\label{Aharmregass}
\abs{\mA(x,\xi)} + \abs{\partial_{x_j} \mA^k(x,\xi)}  + \abs{\xi} \,\abs{\partial_{\xi_j} \mA^k(x,\xi)} \leq \beta \abs{\xi}^{p-1}, \\
(\mA(x,\xi)-\mA(x,\zeta)) \cdot (\xi-\zeta) \geq \delta (\abs{\xi} + \abs{\zeta})^{p-2} \abs{\xi-\zeta}^2.
\end{gather}

If $u \in W^{1,p}(B_1)$ is a weak solution of $\mdiv \mA(x,\nabla u) = 0$ in $B_1$, then there are constants $C > 0$ and $\alpha \in (0,1)$ only depending on $n$, $p$, $\delta$, $\beta$, and $\norm{u}_{W^{1,p}(B_1)}$ such that $u$ is $C^{1+\a}$ regular in any subdomain of $B_1$ and
$$
\norm{u}_{C^{1,\alpha}(\overline{B}_{1/4})} \leq C.
$$
If additionally $\mA$ is $C^r$, $r > 1$, on $B_1\times \R^n \setminus \{0\}$ then $u$ is $C^{r+1}_*$ in any open set where $\nabla u$ is nonvanishing.
\end{prop}
\begin{proof}
If $p > n$, it follows from Sobolev embedding that 
$$
\norm{u}_{L^{\infty}(B_1)} \leq \norm{u}_{C^{1-n/p}(\overline{B}_1)} \leq C_{n,p} \norm{u}_{W^{1,p}(B_1)}.
$$
On the other hand, if $1 < p \leq n$ we have by~\cite[Thms. 1 and 2]{Serrin} and by the bounds $\mA(x,\xi) \cdot \xi \geq \delta \abs{\xi}^p$, $\abs{\mA(x,\xi)} \leq \beta \abs{\xi}^{p-1}$ that 
$$
\norm{u}_{L^{\infty}(B_{1/2})} \leq C_{n,p,\delta,\beta} \norm{u}_{L^p(B_1)}.
$$
The assumptions on $\mA$ also imply that for all $x \in B_1$ and $\xi \in \mR^n$,
\begin{gather}\label{Aharmonic_reg_assumptions}
\sum_{j,k=1}^n \partial_{\xi_j} \mA^k(x,\xi) h_j h_k \geq 2^{p-2} \delta \abs{\xi}^{p-2} \abs{h}^2 \text{ for all $h \in \mR^n$}.
\end{gather}
The $C^{1+\alpha}$ estimate is then a consequence of~\cite[Thm. 2]{DiBenedetto} (see also \cite{Manfredi} for a stronger result where $\mA$ is only assumed to be H\"older continuous).


Suppose then that $\mA$ is $C^r$ regular with $r > 1$, and that $\nabla u \neq 0$ in some open set $U\subset\subset B_1$. We first use Lemma \ref{sobo22} and observe that $u \in W^{2,2}_{loc}(U)$. Since $u$ is $\mA$-harmonic in $U$, we have in the weak sense 
\begin{equation*}
 \p_j \left[ \mA^j(x,\nabla u(x)) \right] = 0.
\end{equation*}
Since $\mA\in C^r$ and $\nabla u \in W^{1,2}(U)$, the chain rule for derivatives holds; see e.g.~\cite[Thm. 6.16]{Lieb}. We have
\begin{equation}\label{lin_elliptic}
 \p_{x_j}\mA^j(x,\nabla u(x))+\p_{\xi_k}\mA^j(x,\nabla u(x))\p_j\p_ku=0.
\end{equation}
By the assumption that $\nabla u\neq 0$ in $U$, in any subdomain $U_1 \subset \subset U$ we have $|\nabla u|\geq\eps>0$ for some positive $\eps$. The inequality~\eqref{Aharmonic_reg_assumptions}  yields in $U_1$ 
\begin{equation*}
 \p_{\xi_k}\mA^j(x,\nabla u(x))h_jh_k\geq 2^{p-2} \d |\nabla u(x)|^{p-2} |h|^2\geq 2^{p-2}\d \eps^{p-2} |h|^2.
\end{equation*}
This shows that the equation~\eqref{lin_elliptic} may be interpreted as a linear elliptic equation in nondivergence form. For clarity we write \eqref{lin_elliptic} as 
\begin{equation}\label{lin_elliptic2}
a^{jk}(x)\p_j\p_k u=f,
\end{equation}
where the coefficients $a^{jk}(x)$ and the function $f(x)$, defined by 
\begin{gather*}
 a^{jk}(x)=\p_{\xi_k}\mA^j(x,\nabla u(x)) \\
f(x)=-\p_{x_j}\mA^j(x,\nabla u(x)),
\end{gather*}
are in the H\"older class $C^\sigma$ with $\sigma=\min(\a,r-1)\in(0,1)$. Also, as mentioned $u\in W^{2,2}_{loc}(U)$. Thus by elliptic regularity $u\in C^{2+\sigma}$; see e.g.~\cite[Appx. J]{Besse}.

If $r$ is large we can bootstrap this argument to show higher regularity. Since $r-1 \geq \sigma$, we can write $r = k_0+1+\sigma+s$ where $k_0$ is a nonnegative integer and $0 \leq s < 1$. We will show that $u \in C^{k_0+2+\sigma}$. If $k_0 = 0$ this has already been proved. If $k_0 \geq 1$ note that $u \in C^{2+\sigma}$, so $a^{jk}$ and $f$ have $C^{1+\sigma}$ regularity and elliptic regularity (Schauder estimates) imply $u \in C^{3+\sigma}$. Continuing this, we obtain $u \in C^{k_0+2+\sigma}$, or equivalently $u \in C^{r+1-s}$. But $a^{jk}, f \in C^{r-1}$, so $u \in C^{r+1}_*$~\cite[Theorem 14.4.3]{Taylor3}.
\end{proof}

We are now ready to show the existence and regularity of $\mA$-harmonic coordinates.

\begin{thm}[$\mathcal{A}$-harmonic coordinates]\label{A-harm-coord}
Let $\Omega$ be an open set in $\mR^n$ and let $1 < p < \infty$. Consider a $C^1$ function $\mA: \Omega \times (\mR^n \times \{0\}) \to \mR^n$ that satisfies for $x \in \Omega$ and $\xi, \zeta \in \mR^n$,
\begin{gather}
\mA(x,t\xi) = t^{p-1} \mA(x,\xi) \text{ for $t \geq 0$}, \label{mA_assumptions1} \\
(\mA(x,\xi)-\mA(x,\zeta)) \cdot (\xi-\zeta) \geq \delta (\abs{\xi} + \abs{\zeta})^{p-2} \abs{\xi-\zeta}^2. \label{mA_assumptions2}
\end{gather}
Given any point $x_0 \in \Omega$, there is $C^1$ (in fact $C^{1+\a}$ for some $\a>0$) diffeomorphism $U$ from some neighborhood of $x_0$ onto an open set in $\mR^n$ such that all coordinate functions of $U$ are $\mA$-harmonic. 

Moreover, if $\mA \in C^r$ with $r > 1$, then $U$ is $C^{r+1}_*$ and any $C^1$ diffeomorphism whose coordinate functions are $\mA$-harmonic has this regularity. Given any invertible matrix $S$ and any $\eps > 0$, there is $U$ defined near $x_0$ such that 
$$
\norm{DU(x_0) - S} < \eps.
$$
\end{thm}
\begin{proof}
Assume first that $S$ is the identity matrix. We normalize matters so that $x_0 = 0$ and $B_1 \subset \Omega$, and consider the equation 
$$
\mdiv \mA(x,\nabla u) = 0 \quad \text{in $B_1$}.
$$
The $\mA$-harmonic coordinates will be obtained as the map $U = (u^1, \ldots, u^n)$, where each $u^j = u^j_{\eps}$ solves for $\eps$ small the Dirichlet problem 
$$
\mdiv \mA(x,\nabla u^j) = 0 \   \   \text{in } B_{\eps}, \quad u^j|_{\partial B_{\eps}} = x^j.
$$

The assumptions on $\mA$ imply that the conditions in Propositions \ref{Aharmonic_dirichlet} and \ref{Aharmonic_regularity} are fulfilled. Then for $0 < \eps \leq 1$, the Dirichlet problem above has a solution $u^j \in W^{1,p}(B_{\eps})$ and additionally $u^j$ has $C^{1+\alpha}$ regularity.  We will show that when $\eps$ is sufficiently small, the Jacobian matrix $DU = (\partial_j u^k)_{j,k=1}^n$ is invertible at the origin. This gives the existence of the required $C^1$ (in fact $C^{1+\alpha}$) diffeomorphism by the inverse function theorem. The additional $C^{r+1}_*$ regularity when $\mA \in C^r$ for any system of $\mA$-harmonic coordinates then follows from the previous proposition.

We will write $u = u^1$ below. Define the dilated coordinates $\tilde{x} = x/\eps$, and let $\tilde{u}(\tilde{x}) = \eps^{-1} u(\eps \tilde{x})$ and similarly for $\tilde{\varphi}$ where $\varphi \in C^{\infty}_c(B_{\eps})$ is a test function. We have $\nabla u(x) = (\nabla \tilde{u})(x/\eps)$, and 
$$
\int_{B_{\eps}} \mA(x,\nabla u) \cdot \nabla \varphi \,dx = 0 \Leftrightarrow \int_{B_1} \mA_{\eps}(\tilde{x}, \nabla \tilde{u}(\tilde{x})) \cdot \nabla \tilde{\varphi}(\tilde{x}) \,d\tilde{x} = 0
$$
where
$$
\mA_{\eps}^j(\tilde{x},\tilde{q}) = \mA^j(\eps \tilde{x}, \tilde{q}).
$$
Note that $\mA_{\eps}$ satisfies the assumptions of this theorem, and consequently of Propositions \ref{Aharmonic_dirichlet} and \ref{Aharmonic_regularity}, uniformly with respect to $0 < \eps \leq 1$. For instance 
\begin{equation} \label{aeps_strict_monotonicity}
(\mA_{\eps}(x,\xi)-\mA_{\eps}(x,\zeta)) \cdot (\xi-\zeta) \gtrsim (\abs{\xi} + \abs{\zeta})^{p-2} \abs{\xi-\zeta}^2
\end{equation}
where the implied constants are independent of $\eps$. The function $\tilde{u}$ solves 
$$
\mdiv \mA_{\eps}(\tilde{x},\nabla \tilde{u}) = 0 \ \ \text{ in } B_1, \quad \tilde{u}|_{\partial B_1} = \tilde{x}^1.
$$
In fact we can take $\tilde{u}$ to be a solution of this equation provided by Proposition \ref{Aharmonic_dirichlet}, satisfying 
\begin{gather*}
\norm{\tilde{u}}_{W^{1,p}(B_1)} \lesssim 1, \\
\norm{\tilde{u}}_{C^{1+\alpha}(\overline{B}_{1/4})} \lesssim 1.
\end{gather*}
The original function $u$ may be obtained from $\tilde{u}$ by scaling.

Write $\tilde{u} = \tilde{u}_0 + \tilde{u}_1$ where $\tilde{u}_0 = \tilde{x}^1$. We will prove that 
\begin{equation} \label{u1_estimate}
\int_{B_1} \abs{\nabla \tilde{u}_1}^p \,d\tilde{x} \lesssim \eps^{\min\{p,p'\}}
\end{equation}
where $1/p + 1/p' = 1$. Define 
$$
I = \int_{B_1} (\abs{\nabla \tilde{u}} + \abs{\nabla \tilde{u}_0})^{p-2} \abs{\nabla \tilde{u}_1}^2 \,d\tilde{x}.
$$
The claim \eqref{u1_estimate} will follow from the estimate 
\begin{equation} \label{I_estimate}
I \lesssim \eps \norm{\nabla \tilde{u}_1}_{L^p}.
\end{equation}
In fact, if $p \geq 2$ then 
\begin{align*}
\int_{B_1} \abs{\nabla \tilde{u}_1}^p \,d\tilde{x} = \int_{B_1} \abs{\nabla \tilde{u}_1}^{p-2} \abs{\nabla \tilde{u}_1}^2 \,d\tilde{x} \leq \int_{B_1} (\abs{\nabla \tilde{u}} + \abs{\nabla \tilde{u}_0})^{p-2} \abs{\nabla \tilde{u}_1}^2 \,d\tilde{x} = I
\end{align*}
and \eqref{I_estimate} implies \eqref{u1_estimate}. In the case $1 < p < 2$, the H\"older inequality implies 
\begin{align*}
\int_{B_1} \abs{\nabla \tilde{u}_1}^p \,d\tilde{x} &= \int_{B_1} (\abs{\nabla \tilde{u}} + \abs{\nabla \tilde{u}_0})^{-p(p-2)/2} (\abs{\nabla \tilde{u}} + \abs{\nabla \tilde{u}_0})^{p(p-2)/2} \abs{\nabla \tilde{u}_1}^p \,d\tilde{x} \\
&\leq \left( \int_{B_1} (\abs{\nabla \tilde{u}} + \abs{\nabla \tilde{u}_0})^p \,d\tilde{x} \right)^{\frac{2-p}{2}} I^{p/2} \lesssim I^{p/2}
\end{align*}
since $\norm{\tilde{u}}_{W^{1,p}(B_1)} \lesssim 1$. Again \eqref{I_estimate} implies \eqref{u1_estimate}.

To show \eqref{I_estimate} we use \eqref{aeps_strict_monotonicity} and compute 
\begin{align*}
I &= \int_{B_1} (\abs{\nabla \tilde{u}} + \abs{\nabla \tilde{u}_0})^{p-2} \abs{\nabla \tilde{u}_1}^2 \,d\tilde{x} \\
 &\lesssim \int_{B_1} (\mA_{\eps}(\tilde{x},\nabla \tilde{u}) - \mA_{\eps}(\tilde{x},\nabla \tilde{u}_0)) \cdot (\nabla \tilde{u} - \nabla \tilde{u}_0) \,d\tilde{x}.
\end{align*}
Since $\tilde{u}$ solves the $\mA_{\eps}$-harmonic equation and $\tilde{u}_0$ solves that $\mA_0$-harmonic equation (since $\nabla\tilde{u}_0$ is constant), and since $\tilde{u} - \tilde{u}_0 \in W^{1,p}_0(B_1)$ can be used as a test function, we obtain 
\begin{align*}
 &\int_{B_1} (\mA_{\eps}(\tilde{x},\nabla \tilde{u}) - \mA_{\eps}(\tilde{x},\nabla \tilde{u}_0)) \cdot (\nabla \tilde{u} - \nabla \tilde{u}_0) \,d\tilde{x} \\
 &= -\int_{B_1} \mA_{\eps}(\tilde{x},\nabla \tilde{u}_0) \cdot (\nabla \tilde{u} - \nabla \tilde{u}_0) \,d\tilde{x} \\
  &= -\int_{B_1} (\mA_{\eps}(\tilde{x},\nabla \tilde{u}_0) - \mA_0(\tilde{x},\nabla \tilde{u}_0)) \cdot (\nabla \tilde{u} - \nabla \tilde{u}_0) \,d\tilde{x} \\
 &\lesssim \norm{\mA_{\eps}(\,\cdot\,,\nabla \tilde{u}_0) - \mA_0(\,\cdot\,,\nabla \tilde{u}_0)}_{L^{p'}(B_1)} \norm{\nabla \tilde{u}_1}_{L^p(B_1)}.
\end{align*}
Here, since $\mA$ is Lipschitz continuous and since $\nabla \tilde{u}_0 =e_1$ is constant, we have 
$$
\abs{\mA_{\eps}(\tilde{x},\nabla \tilde{u}_0) - \mA_0(\tilde{x},\nabla \tilde{u}_0)} = \abs{\mA(\eps \tilde{x}, \nabla \tilde{u}_0) - \mA(0,\nabla \tilde{u}_0)} \lesssim \eps
$$
in $B_1$ and therefore 
$$
\norm{\mA_{\eps}(\,\cdot\,,\nabla \tilde{u}_0) - \mA_0(\,\cdot\,,\nabla \tilde{u}_0)}_{L^{p'}(B_1)} \lesssim \eps.
$$
The estimate \eqref{u1_estimate} follows.

Collecting the results so far, we know that
$$
\norm{\nabla \tilde{u} - \nabla \tilde{u}_0}_{L^p(B_1)} \lesssim \eps^{\min\{1,\frac{1}{p-1}\}}
$$
and 
$$
\norm{\nabla \tilde{u} - \nabla \tilde{u}_0}_{C^{\alpha}(\overline{B}_{1/4})} \lesssim 1.
$$
We ''interpolate'' these estimates by using Lemma~\ref{interpol}, which shows that 
\begin{equation*} \label{u_u0_w1inftyestimate}
\norm{\nabla \tilde{u} - \nabla \tilde{u}_0}_{L^{\infty}(B_{1/8})} = o(1) \quad \text{as } \eps \to 0.
\end{equation*}
Repeating this for all components of $\tilde{U} = (\tilde{u}^1, \ldots, \tilde{u}^n)$ implies that 
$$
\norm{D\tilde{U}(0) - Id} = o(1) \quad \text{as } \eps \to 0.
$$
Thus in particular $D\tilde{U}(0)$ is invertible for small $\eps$, so also $DU(0) = D\tilde{U}(0)$ is invertible (and arbitrarily close to the identity matrix) for small $\eps$.

The inverse function theorem shows that $U$ is a $C^1$ diffeomorphism near $0$, and also $U$ is $C^{r+1}_*$. To obtain this regularity also for the inverse, write
$$
D(U^{-1})= (DU\circ U^{-1})^{-1}.
$$
Since $DU\in C^{r}_*$ and $U^{-1}\in C^{r}$ (by the standard inverse function theorem and the fact that $C^{r+1}_*\subset C^r$), and since $C^r_*$ is an algebra with respect to the usual multiplication of functions, this shows that $U^{-1}\in C^{r+1}_*$.

The proof of the existence of $\mA$-harmonic coordinates is completed in the case where $S = I$. If $S$ is any invertible matrix, we can choose $u^j$ to be  solutions of 
$$
\mdiv \mA(x,\nabla u^j) = 0 \   \   \text{in } B_{\eps}, \quad u^j|_{\partial B_{\eps}} = (Sx)^j.
$$
Repeating the argument results in $\mA$-harmonic coordinates for which $DU(0)$ is arbitrarily close to the matrix $S$.
\end{proof}

The existence of $p$-harmonic coordinates is an immediate consequence.

\begin{proof}[Proof of Theorem \ref{p-harm-coord}]
As discussed above, in any local coordinate system where the metric has $C^r$ regularity, the $p$-harmonic equation is of the form $\mdiv\,\mA(x,\nabla u) = 0$ where $\mA$ is given by \eqref{A_def}, satisfies \eqref{mA_assumptions1} and \eqref{mA_assumptions2}, and has $C^r$ regularity. The existence and regularity of $p$-harmonic coordinates is a consequence of Theorem \ref{A-harm-coord}.

Suppose that the metric in original coordinates at $x_0$ is represented by the matrix $G_0$, and let $U$ be the corresponding $\mA$-harmonic coordinates. The metric in $p$-harmonic coordinates at $x_0$ is given by 
$$
DU(0)^{-t} G_0 DU(0)^{-1}.
$$
Choosing $U$ so that $DU(0)$ is very close to $G_0^{1/2}$, we can arrange that the metric in $p$-harmonic coordinates is arbitrarily close to identity at $x_0$.
\end{proof}

\begin{remark}
It is also possible to consider $\mA$-harmonic coordinates on any Riemannian manifold $(M,g)$. In this case $\mA$ is a bundle map $T^*M\to T^*M$ mapping each fiber $T_{x_0}^*M$ to itself, and the $\mA$-harmonic equation is defined naturally as
\begin{equation*}
 \d \mA(du)=0.
\end{equation*}
The conditions \eqref{mA_assumptions1} and \eqref{mA_assumptions2} are replaced by  
\begin{gather*}
\mA(t\xi) = t^{p-1} \mA(\xi) \text{ for $t \geq 0$}, \\
\langle \mA(\xi)-\mA(\zeta),\xi-\zeta \rangle \geq \delta (\abs{\xi} + \abs{\zeta})^{p-2} \abs{\xi-\zeta}^2,
\end{gather*}
where $\xi, \eta \in T^* M$ and $\langle \,\cdot\,,\,\cdot\, \rangle$ and $\lvert\,\cdot\,\rvert$ are the Riemannian inner product and norm. The existence of $\mA$-harmonic coordinates in this setting follows by a straightforward modification of the previous arguments. We omit the details.
\end{remark}


Standard harmonic coordinates have the property that the Riemannian metric has maximal regularity in these coordinates; there is no coordinate system where the coordinate representation of the Riemannian metric is more regular than in harmonic coordinates. The same is true for $\mathcal{A}$-harmonic coordinates. More generally, we have the analogue of~\cite[Corollary 1.4]{Kazdan_Warner} with exactly the same proof.


\begin{prop}\label{max_reg}
 Let a Riemannian metric $g$ be $C^r$ regular, $r>0$, in some local coordinates $\varphi$ near a point $x_0 \in M$. If a tensor field $T$ is of class $C^s$, $s\geq r$, in the $\varphi$ coordinates, then it is of class at least $C^r_*$ in any $\mA$-harmonic coordinates near $x_0$.
\end{prop}
\begin{proof}
The chart $\varphi$ is $C^{\infty}$ since $M$ is a $C^{\infty}$ manifold, and by Theorem~\ref{A-harm-coord} any $\mathcal{A}$-harmonic coordinates $\psi$ near $x_0$ are necessarily $C^{r+1}_*$ regular. The coordinate representations $T_{\varphi}$ and $T_{\psi}$ of $T$ in the two coordinate systems are related by the pullback $T_{\psi} = (\varphi^{-1} \circ \psi)^* T_{\varphi}$. Since the pullback of a tensor field involves taking first derivatives of the coordinate functions, $T$ is $C^r_*$ in the $\psi$-coordinates.
\end{proof}

In the next section we study the regularity of conformal mappings by using $n$-harmonic coordinates, that is, $p$-harmonic coordinates with $p = \dim(M)$. Recall that a Riemannian manifold is said to be locally conformally flat if near any point there is a smooth coordinate chart where
\begin{equation*}
g_{jk}=c\, \d_{jk}
\end{equation*}
for some positive function $c$. Any two-dimensional manifold is locally conformally flat due to the existence of isothermal coordinates. In dimensions $n \geq 3$, a necessary and sufficient condition for sufficiently regular Riemannian metrics to be locally conformally flat is that the Cotton tensor vanishes for $n=3$ and that the Weyl tensor vanishes for $n\geq 4$~\cite[Theorem 4.24]{Aubin}. 

By the following result, coordinate charts which satisfy $g_{jk} = c \delta_{jk}$ are necessarily $n$-harmonic. 
\begin{prop}\label{n-isothermal}
If $(M,g)$ is a Riemannian manifold, any coordinate chart in which $g_{jk}=c\,\d_{jk}$ is necessarily $n$-harmonic.



\end{prop}
\begin{proof}
If $(x^1,\ldots,x^n)$ is such a coordinate chart, it follows by a direct computation that $\delta(\abs{du}^{n-2} du) = 0$ for $u = x^l$ with any $l$. (It would enough to assume that $c$ is bounded and measurable.)
\end{proof}

\section{Regularity of $C^1$ conformal mappings}

The first application of $n$-harmonic coordinates is to the regularity of $C^1$ conformal mappings between manifolds with $C^r$ ($r > 1$) metric tensors. This follows from Theorem \ref{p-harm-coord} and the well-known fact that the pullback of an $n$-harmonic function by a conformal mapping remains $n$-harmonic. For $C^1$ mappings this fact is an easy consequence of the chain rule. In the next section we will consider the more general case of $1$-quasiregular maps that have only $W^{1,n}_{loc}$ regularity.

\begin{thm} \label{thm_c1_conformal_regularity}
Let $(U,g)$ and $(V,h)$ be Riemannian manifolds such that $g,h \in C^r$, $r > 1$. If $\phi: U \to V$ is a $C^1$ diffeomorphism that is conformal in the sense that $\phi^* h = cg$ for some continuous positive function $c$, then $\phi$ is a $C^{r+1}$ diffeomorphism from $U$ onto $V$.
\end{thm}
\begin{proof}
Choose a point $x_0 \in U$, and let $v$ be $n$-harmonic coordinates near $\phi(x_0)$ in $(V,h)$ provided by Theorem \ref{p-harm-coord}. The map $v$ has $C^{r+1}_*$ regularity. Define $u = \phi^* v$, so that $u$ is a $C^1$ diffeomorphism from a neighborhood of $x_0$ into $\mR^n$.

We claim that all components $u^l$ of $u$ are $n$-harmonic with respect to the metric $g$. To see this, note that when $c$ is a positive function and $\omega$ is a $1$-form and all quantities are continuous, we have in the weak sense 
\begin{gather*}
\delta_{cg} \omega = c^{-n/2} \delta_g (c^{\frac{n-2}{2}} \omega).
\end{gather*}
Using the chain rule $du^l = d(\phi^* v^l) = \phi^* dv^l$, we compute 
\begin{align*}
\delta_g(\abs{du^l}_g^{n-2} du^l) &= \delta_{c^{-1} \phi^* h} ( \abs{\phi^* dv^l}_{c^{-1} \phi^* h}^{n-2} \phi^* dv^l ) \\
 &= c^{n/2} \delta_{\phi^* h} (c^{-\frac{n-2}{2}} ( c \langle \phi^* dv^l, \phi^* dv^l \rangle_{\phi^* h})^{\frac{n-2}{2}} \phi^* dv^l ) \\
 &= c^{n/2} \phi^* (\delta_h (\abs{dv^l}_h^{n-2} dv^l) ) = 0.
\end{align*}
Thus $u$ is an $n$-harmonic coordinate system near $x_0$, and by Theorem \ref{p-harm-coord} it has $C^{r+1}_*$ regularity.

If $r > 1$ is not an integer, we have already proved that $\phi = v^{-1} \circ u$ is a $C^{r+1}$ map near $x_0$. The case where $r$ is an integer is handled as in \cite{Calabi, Taylor} and is based on the following identity for Christoffel symbols $\Gamma_{jk}^l$ of $(U,g)$ and $\tilde{\Gamma}_{jk}^l$ of $(V,h)$ under changes of coordinates:
$$
\Gamma_{ij}^k \partial_{k} \phi^m - \tilde{\Gamma}_{kl}^m \partial_{i} \phi^k \partial_{j} \phi^l = \partial_{i} \partial_{j} \phi^m.
$$
In our case $\phi = v^{-1} \circ u$ is $C^{r+1-\eps}$ for any $\eps > 0$, so $\phi$ is at least $C^2$ and the identity is valid. This shows that the second derivatives of $\phi$ are $C^{r-1}$, and consequently $\phi$ is $C^{r+1}$.
\end{proof}

\section{Regularity of $1$-quasiregular mappings}



In this section we consider the case of $1$-quasiregular mappings between manifolds with $C^r$ metric tensors. The main result, Theorem~\ref{mainthm}, shows that such a mapping is always a $C^{r+1}$ local conformal diffeomorphism. In the formulation and in the proof of the theorem we use the theory of quasiregular mappings on $\R^n$ and on Riemannian manifolds. There are several ways to define quasiregular mappings on  manifolds \cite{HK, Iwaniec}, and we will follow the approach discussed in~\cite{Liimatainen}. Let us first recall the definition of quasiregular mappings on $\R^n$.
\begin{definition}
 A mapping $\phi:\Omega\to \R^n$ is said to be \emph{$K$-quasiregular} if it is of Sobolev class $\W(\Omega,\R^n)$, its Jacobian determinant has constant sign a.e. on $\Omega$ and if it satisfies the \emph{distortion inequality}
\begin{equation*}
\norm{D\phi}_{op}^n\leq K\, J_\phi \mbox{ a.e.}
\end{equation*}
If $\phi$ is in addition a homeomorphism, it is called \emph{$K$-quasiconformal}.
\end{definition}
In the definition $D\phi\in L^n_{loc}(\Omega,\R^{n\times n})$, called the weak differential, is the Jacobian matrix of $\phi$ consisting of the weak derivatives of the component functions $\phi^i$. The norm $\norm{\,\cdot\,}_{op}$ is the operator norm and $J_\phi = \det {D\phi} \in L^1_{loc}(\Omega)$ is the Jacobian determinant of $\phi$. 

It can be seen from the definition that a $1$-quasiregular mapping $\phi$ satisfies
\begin{equation*}
 D\phi^TD\phi=c\, I_{n\times n} \mbox{ a.e.}
\end{equation*}
Here $c$ is some a.e.~positive function. Thus $1$-quasiregular mappings are conformal mappings in this sense. 
For the theory of quasiregular and quasiconformal mappings on $\R^n$ we refer to~\cite{Iwaniec,Reshetnyak,Vaisala, Rickman}

Since we are considering mappings between Riemannian manifolds, it is natural to consider the Riemannian analogue of quasiregular mappings. For our purposes the natural class of mappings is the class of Riemannian quasiregular mappings introduced in~\cite{Liimatainen}. 
\begin{definition}\label{quasiregularity}
Let $\phi: (M,g)\rightarrow (N,h)$ be a localizable $\W(M,N)$ mapping between Riemannian manifolds with continuous Riemannian metrics. In this case, the mapping $\phi$ is said to be \emph{Riemannian $K$-quasiregular} if the Jacobian determinant of $\phi$ has locally constant sign and if it satisfies the distortion inequality
\begin{equation*}
 ||D\phi||_g^n\leq K\, \Detg{D\phi}{g} \mbox{ a.e.}
\end{equation*}
If the mapping $\phi$ is in addition a homeomorphism, it is called \emph{Riemannian $K$-quasiconformal}.
\end{definition}

Above $D\phi: TM\to TN \in L^n_{loc}$ denotes the weak differential of $\phi$. The (pointwise) norm $||D\phi||_g$ and the determinant $\Detg{D\phi}{g}$ of $D\phi$ are invariant quantities defined by the formulas 
\begin{align*}
 ||D\phi||_g& =\frac{1}{n^{1/2}}\tr{g^{-1}\phi^*h}^{1/2}, \\
  \Detg{D\phi}{g} &= \det{g^{-1}\phi^*h}^{1/2}.
\end{align*}
A mapping is localizable if for every $p\in M$ there is an open set $U$ containing $p$ and a coordinate neighborhood $V$ of $\phi(p)$ such that $\phi(U)\subset\subset V$. The condition of localizability is satisfied for example if $\phi$ is continuous. The Jacobian determinant of $\phi$ has locally constant sign if the Jacobian determinant of every coordinate representation of $\phi$ is either non-negative or non-positive a.e.

Riemannian quasiregular mappings share the same analytic properties as quasiregular mappings on $\R^n$, but also take the Riemannian geometry of the manifolds naturally into account. In the case that a mapping $\phi:(M,g)\to (N,h)$ is Riemannian $1$-quasiregular, the mapping satisfies a.e. the equation
\begin{equation*}
 \phi^*h=c\,g,
\end{equation*}
where the pullback is defined in terms of the weak differential, for some a.e. positive function $c$. Note also that a differentiable mapping satisfying this equation is $1$-quasiregular. For details of these statements and for further details of the definition of Riemannian quasiregular mappings see~\cite{Liimatainen} (see also~\cite{dissertation}). 

We record the basic properties of Riemannian quasiregular mappings in the following lemma. We omit the proof and refer to~\cite{Liimatainen, dissertation}.
\begin{thm}\label{rqr}
 Let a mapping $\phi:(M,g)\to (N,h)$ be Riemannian $K$-quasiregular and assume that $g$ and $h$ are continuous. Then the following hold:
\begin{enumerate}
\item The mapping $\phi$ is differentiable a.e., and at the points where the differential exists, it coincides with the weak differential. The mapping $\phi$ can be redefined on a set of measure zero to be continuous.
\item Let $u\in \W(N)$. Then $\phi^*u=u\circ \phi$ is of Sobolev class $\W(N)$ and $\phi^*u$ satisfies a.e. the chain rule of derivation:
\begin{equation*}\label{chain_rule}
 \p_i(u\circ\phi)=\p_a u|_\phi \p_i\phi^a \mbox{ i.e. }  d\phi^*u=\phi^*du.
\end{equation*}
Moreover, 
\begin{equation*}
|d(u\circ\phi)|^n_g\leq n^{n} K\, \Det{D\phi}\phi^*(|du|_h^n) \mbox{ a.e.}
\end{equation*}
\item If the mapping $\phi$ is non-constant, the Riemannian Jacobian determinant
\begin{equation*}
 \Det{D\phi}=\sqrt{\det{g^{-1}\phi^*h}}\in L^1_{loc}(M)
\end{equation*}
of $\phi$ is non-vanishing a.e.
\item If $\phi$ is in addition a homeomorphism, and thus $K$-quasiconformal, its inverse is $K^{n-1}$-quasiconformal. Also, integration by substitution is valid:
\begin{equation*}
 \int_M f\circ \phi\, \Det{D\phi}d\mu_g=\int_{N} f d\mu_h.
\end{equation*}
Here $f$ is any integrable function on $N$.
\end{enumerate}
\end{thm}

Note that the lemma essentially states that all the standard formulas for derivatives and integration of smooth mappings continue to hold also for Riemannian quasiregular (or quasiconformal) mappings.

We will also use the following auxiliary result.
\begin{lemma}\label{1eps}
Let $\phi:(M,g)\to (N,h)$ be a Riemannian $1$-quasiregular mapping. Let $p\in M$. Given any $\eps > 0$, there is $\delta > 0$ such that if $U$ and $V$ are coordinate neighborhoods of $p$ and $\phi(p)$ with $\phi(U)\subset V$ and
$$
\norm{g_{jk}(p)-I_{n\times n}}_{op}\leq \delta, \qquad \norm{h_{jk}(\phi(p)) - I_{n \times n}}_{op} \leq \delta,
$$
then there are subdomains $U'\subset\subset U$ and $V'\subset\subset V$ such that the coordinate representation of $\phi:U'\to V'$ is an $(1+\e)$-quasiregular mapping on $\R^n$.
\end{lemma}
\begin{proof}
We will use local coordinate expressions for $g$, $h$, and $\phi$. If $\delta$ is chosen sufficiently small, there are subdomains $U'$ and $V'$ such that 
\begin{equation}
 \frac{||g||_{op}^n}{\det{g}} \leq (1+\epsilon)^{1/2}    \mbox{ and } \frac{||h^{-1}||_{op}^n}{\det{h^{-1}}}\leq (1+\epsilon)^{1/2} 
\end{equation}
on $U'$ and $V'$ respectively. Using the fact that
\begin{equation}
 \phi^*h=cg \mbox{ a.e.}
\end{equation}
we have a.e.
\begin{align*}
\abs{(D\phi) X}^2&=\abs{h(\phi)^{-1/2} h(\phi)^{1/2} (D\phi)X}^2 \leq ||h(\phi)^{-1/2}||_{op}^2 \langle h(\phi)(D\phi) X, (D\phi) X \rangle \\
 &\leq \norm{h(\phi)^{-1}}_{op} \langle cgX, X \rangle \leq c \norm{h(\phi)^{-1}}_{op} \norm{g}_{op} |X|^2.
\end{align*}
We also have $\det{D\phi}=c^{n/2}\det{h^{-1}(\phi)}^{1/2} \det{g}^{1/2}$ a.e. yielding
\begin{align*}
 \frac{||D\phi||_{op}^n}{J_\phi}\leq \frac{c^{n/2} \norm{h^{-1}(\phi)}_{op}^{n/2} \norm{g}_{op}^{n/2}}{c^{n/2}\det{h^{-1}(\phi)}^{1/2} \det{g}^{1/2}}\leq 1+\epsilon \mbox{ a.e.}
\end{align*}

\end{proof}

We continue with the following proposition, which states that the composition of an $n$-harmonic function and a conformal mapping gives an $n$-harmonic function. The proof is a straightforward calculation that uses the weak formulation of $n$-harmonicity, but the proof is somewhat lengthy due to the weak regularity assumptions. 

\begin{prop}\label{pullback_of_nh}
Let $\phi:(M,g)\rightarrow (N,h)$ be Riemannian $1$-quasiconformal and assume that the Riemannian metrics $g$ and $h$ are of class $C^r$, $r>1$. In this case, the pullback $\phi^*v=v\circ\phi$ of any $n$-harmonic function $v$ on $N$ is $n$-harmonic.
\end{prop}

\begin{proof}
Let $v$ be an $n$-harmonic function on $N$. We write for simplicity $u=v\circ\phi$ and show that $u$ is $n$-harmonic. The function $v$ is at least $C^1$ by Proposition~\ref{Aharmonic_regularity} and thus $u\in \W(M)$ by Theorem~\ref{rqr}. The function $u$ is $n$-harmonic if the functional $F$, 
\begin{equation}\label{weak_composite_form}
 F(w):=\int_M g(|du|_g^{n-2}du,dw) \,dV_g,
\end{equation}
vanishes on the space $C^{\infty}_c(M)$ of smooth functions with compact support in $M$.

It is sufficient to show that $F$ vanishes for all functions of the form $f\circ \phi$, $f\in C_c^\infty(N)$. To see this, let $w\in C^{\infty}_c(M)$, $K=\supp{w}$, and assume that $F$ vanishes for all functions of the form $f\circ\phi$, $f\in C_c^\infty(N)$. The function $k:=w\circ\phi^{-1}$ is of Sobolev class $W^{1,n}_0(N)$ by Theorem~\ref{rqr} and by the fact that $k$ is supported in the compact set $\phi(K)$.

Let us approximate $k$ by $C_c^\infty(N)$ functions $k_i\to k$ on $W^{1,n}(N)$. By H\"older's inequality and the fact that by our assumption $F(k_i\circ\phi)=0$, we have
\begin{align*}
 |F(w)|&\leq |F((k_i-k)\circ\phi)|+|F(k_i\circ\phi)| \\
  &\leq ||du||_{L^n(M)}^{\frac{n-1}{n}}\l(\int_M |d((k_i-k)\circ\phi)|_g^n \,dV_g\r)^{1/n}.
\end{align*}
Moreover, the following calculation is valid by results of Theorem~\ref{rqr},
\begin{align*}
 \int_M |d((k_i-k)&\circ\phi)|^n_gdV_g =\int_M \frac{|d((k_i-k)\circ\phi)|^n_g}{\Det{D\phi}}\, \Det{D\phi} \,dV_g \\
  & \leq n^{n} \int_M (|d(k_i-k)|_h^n)\circ\phi\, \Det{D\phi} \,dV_g \\
  &= n^{n}\int_{\phi(M)} |d(k_i-k)|_h^n \,dV_h  \leq n^{n} ||k_i-k||_{W^{1,n}(N)}^n.
\end{align*}
Since we have $||k_i-k||_{W^{1,n}(N)}^n\to 0$, $i\to\infty$, it follows that $F(w)=0$.

We use results of Theorem~\ref{rqr} again to evaluate~\eqref{weak_composite_form} for $w=f\circ \phi$, where $f\in C_c^\infty(N)$.  We have
\begin{align}\label{neval}
  \notag \int_M &g(|du|_g^{n-2}du,d(f\circ\phi))dV_g = \int_M g(|du|_g^{n-2}du,d(f\circ\phi)) \frac{\Det{D\phi}}{\Det{D\phi}}dV_g \\ 
  &= \int_{\phi(M)} |du|_g^{n-2}|_{\phi^{-1}} g(du,d(f\circ\phi))|_{\phi^{-1}} \frac{1}{\Det{D\phi}|_{\phi^{-1}}} dV_h.
\end{align}
By the remarks about Riemannian $1$-quasiregular mappings, $\phi$ satisfies
\begin{equation}\label{conf}
 \phi^*h=cg \mbox{ a.e.}
\end{equation}
Therefore the Jacobian determinant $\Det{D\phi}$ equals $c^{n/2}$. By Theorem~\ref{rqr}, the Jacobian matrix $D\phi$ of $\phi$ is invertible a.e. and the chain rule for derivatives is valid a.e. yielding
\begin{align*}
 g(du,d(f\circ\phi))&=g^{ij} \p_i (v\circ\phi) \p_j(f\circ\phi) = (\p_i\phi^a)\, g^{ij}\, (\p_j \phi^b)\, (\p_a v|_\phi\, \p_b f|_\phi) \\
  & =  c\,(D\phi)_i^a  (D\phi^T h|_\phi  D\phi)^{ij} (D\phi)_j^b\, (\p_a v|_\phi\, \p_b f|_\phi)  \\
  &=c\, (h^{ab} \p_a v\, \p_b f)\circ\phi  = c\, (h(dv,df))\circ\phi.
\end{align*}
Here we have also used~\eqref{conf} in the third equality. By a similar calculation we get an analogous expression for $g(du,du)$. Thus, the right hand side of~\eqref{neval} equals
\begin{equation*}
 \int_{\phi(M)} c|_{\phi^{-1}}^{(n-2)/2+1-n/2}|dv|_h^{n-2} h(dv,df)dV_h=\int_{\phi(M)} |dv|_h^{n-2} h(dv,df)dV_h=0.
\end{equation*}
We conclude that $u$ is $n$-harmonic.
\end{proof}

Our main theorem is a characterization of conformal mappings between Riemannian manifolds. The characterization is slightly different when $n=2$ and when $n\geq 3$. We begin with the latter case.
\begin{thm}\label{mainthm}
Let $(M,g)$ and $(N,h)$ be Riemannian manifolds, $n\geq 3$, with $g, h\in C^r$, $r>1$. Let $\phi: M\rightarrow N$ be a non-constant mapping. Then the following are equivalent:
\begin{align}
 &\phi  \mbox{ is a Riemannian $1$-quasiregular mapping}, \label{1qc_anal} \\
 &\phi  \mbox{ is locally bi-Lipschitz and } \phi^*h=c\,g \mbox{ a.e.,}  \label{weak_form} \\ 
 &\phi  \mbox{ is a local $C^1$ diffeomorphism and } \phi^*h=c\,g, \label{1dif} \\
 &\phi  \mbox{ is a local $C^{r+1}$ diffeomorphism and } \phi^*h=c\,g. \label{rdif}
\end{align}
\end{thm}
\begin{proof}
I. We show that~\eqref{1qc_anal} implies~\eqref{rdif}. First we establish some topological properties of $\phi$. Let $p\in M$ and let $U$ be a neighborhood of $p$ and let $v=(v_1,\ldots,v_n)$ be $C^{r+1}_*$ regular $n$-harmonic coordinates around $\phi(p)$ such that $\phi(U)\subset V$. 

Since $g$ and $h$ are continuous, the coordinate representation of $\phi$ satisfies the $\R^n$--definition of a quasiregular mapping~\cite{Liimatainen}. We denote the coordinate representation of $\phi$ simply by $\phi:U\to V$, where $U$ and $V$ are open in $\R^n$. By Theorem 8.13.1 of~\cite{Iwaniec}, there is an $\e>0$ such that every $(1+\e)$-quasiregular mapping $\phi:U\to V$ is locally injective. According to Theorem~\ref{p-harm-coord} and Lemma~\ref{1eps} we can reduce $U$ and $V$ such that $\phi:U\to V$ is $(1+\e)$-quasiregular mapping on $\R^n$. Therefore, by Theorem~\ref{rqr} and by the remarks above, we can assume that $\phi:U\to V$ is a homeomorphism and thus quasiconformal (see e.g.~\cite[Thm. 7.7.1, Thm. 16.12.1]{Iwaniec}).

Proposition \ref{pullback_of_nh} applied to $\phi: (U,g) \to (V,h)$ shows that the pullback $u = \phi^* v$ of the $n$-harmonic coordinates $v=(v_1,\ldots,v_n)$ on $V$ gives an $n$-tuple, $u=(u^1,\ldots u^n)$, of $n$-harmonic functions on $U$ and we can express $\phi:U\to V$ as
\begin{equation*}
 \phi = v^{-1}\circ u.
\end{equation*}
In particular, since $v$ is a $C^1$ diffeomorphism and since $n$-harmonic functions are $C^1$ regular, we see that $\phi$ is a $C^1$ map. By Theorem \ref{rqr} the inverse of $\phi:(U,g)\to (V,h)$ is also Riemannian $1$-quasiconformal and thus we can apply the same argument to show that $\phi^{-1}$ is $C^1$. Then $\phi: U \to V$ is a $C^1$ diffeomorphism, and in particular its derivative is invertible. Since $u = v \circ \phi$ we see that the gradient of each $u^j$ is nonvanishing in $U$, and the regularity of $n$-harmonic functions (Proposition \ref{Aharmonic_regularity}) shows that $u$ and also $\phi$ are actually $C^{r+1}$ regular at least as long as $r$ is not an integer. The case $r\in \N$ is handled by the exact same argument as in the proof of Theorem \ref{thm_c1_conformal_regularity}.

II. One trivially has \eqref{rdif} $\implies$ \eqref{1dif} $\implies$ \eqref{weak_form}. 
Finally, a local bi-Lipschitz mapping $M\rightarrow N$ is localizable as a continuous mapping. Also its Jacobian determinant has locally constant sign, since the local degree of $\phi$ is constant and equals $J_\phi$ a.e. (cf.~\cite[Thm. 3.3.4]{AstalaIwaniecMartin}). Since locally the Euclidean norm and the distance metric $d$ are equivalent, the coordinate representation of $\phi$ is a locally bi-Lipschitz mapping on $\R^n$ and thus belongs to the Sobolev class $W^{1,p}_{loc}$ for all $p>1$. Thus \eqref{weak_form} implies \eqref{1qc_anal}.
\end{proof}

In the previous theorem we considered both topological and regularity properties of Riemannian $1$-quasiregular mappings simultaneously. We observed that even without a local injectivity assumption the mapping actually is a local diffeomorphism. If we assume that the mapping is a priori topologically a homeomorphism, we have the following statement that is also valid (and well known) in two dimensions. The proof is analogous to that of the previous theorem.

\begin{thm}\label{mainthm2}
Let $(M,g)$ and $(N,h)$ be Riemannian manifolds, $n\geq 2$, with $g, h\in C^r$, $r>1$. Let $\phi: M\rightarrow N$ be a homeomorphism. The following are equivalent:
\begin{align*}
 &\phi  \mbox{ is a Riemannian $1$-quasiconformal mapping}, \\
 &\phi  \mbox{ is locally bi-Lipschitz and } \phi^*h=c\,g \mbox{ a.e.,} \\ 
 &\phi  \mbox{ is a $C^1$ diffeomorphism and } \phi^*h=c\,g, \\
 &\phi  \mbox{ is a $C^{r+1}$ diffeomorphism and } \phi^*h=c\,g.
\end{align*}
\end{thm}

\begin{remark}
In two dimensions we cannot expect that $1$-quasiregular mappings are locally invertible, as shown by the map $z\mapsto z^2$ in the complex plane.
\end{remark}

Let us conclude by a simple consequence of the previous results. We call a mapping $(M,g)\to (N,h)$ conformal if it is a homeomorphism and satisfies one (and thus all) of the conditions~\eqref{1qc_anal}--~\eqref{rdif}.

\begin{cor}
 Let $(M,g)$ and $(N,h)$ be Riemannian manifolds with $C^r$, $r>1$, Riemannian metrics. Let $(\phi_j)$ be a sequence of conformal mappings $(M,g)\to (N,h)$ converging uniformly on compact sets to a mapping $\phi:M\to N$ . Then $\phi$ is a $C^{r+1}$ conformal mapping.
\end{cor}
\begin{proof}
A conformal mapping is Riemannian $1$-quasiregular. It is well known (see~\cite{Liimatainen}) that the sequence $(\phi_j)$ converges to a Riemannian $1$-quasiregular mapping $\phi$. By the previous results, $\phi$ is a $C^{r+1}$ conformal mapping.
\end{proof}



\appendix
\section{Auxiliary results} \label{sec_appendix}

We first give a simple interpolation lemma.

\begin{lemma}[Interpolation]\label{interpol}
Let $\Omega$ be a bounded open set in $\mR^n$, and suppose that $f \in C^{\alpha}(\overline{\Omega})$ where $0 < \alpha \leq 1$. Let also $K \subset \Omega$ be a compact set, let $\delta_0 = \text{dist}(K, \mR^n \setminus \Omega)$, and let $M > 0$. If $1 \leq p < \infty$ and if 
\begin{gather*}
[f]_{C^{\alpha}(\overline{\Omega})}\leq M, \\
\norm{f}_{L^p(\Omega)} \leq \delta_0^{\frac{n+\alpha p}{p}} M,
\end{gather*}
then 
$$
\norm{f}_{L^{\infty}(K)} \leq C_{n,p,\alpha} M^{\frac{n}{n+\alpha p}} \norm{f}_{L^p(\Omega)}^{\frac{\alpha p}{n+\alpha p}}.
$$
Here $[f]_{C^{\alpha}(\overline{\Omega})} = \sup \left\{ \frac{\abs{f(x)-f(y)}}{\abs{x-y}^{\alpha}} \,;\, x, y \in \overline{\Omega}, x \neq y \right\}$.
\end{lemma}
\begin{proof}
Assume that $f$ is not constant (the lemma is always true if $f$ is constant), and choose a point $x_0 \in K$. Write $B_{\delta} = B_{\delta}(x_0)$. For $0 < \delta \leq \delta_0$, we compute 
\begin{align*}
\norm{f}_{L^p(B_{\delta})} &\geq \norm{f(x_0)}_{L^p(B_{\delta})} - \norm{f - f(x_0)}_{L^p(B_{\delta})} \\
 &= \abs{f(x_0)} \abs{B_{\delta}}^{1/p} - \left( \int_{B_{\delta}} \abs{f(x)-f(x_0)}^p \,dx \right)^{1/p} \\
 &\geq c_{n,p} \delta^{n/p} \abs{f(x_0)} - [f]_{C^{\alpha}(\overline{\Omega})} \left( \int_{\abs{x} < \delta} \abs{x}^{\alpha p} \,dx \right)^{1/p} \\
 &\geq c_{n,p} \delta^{n/p} \abs{f(x_0)} - C_{n,p,\alpha} \delta^{\alpha+n/p} M.
\end{align*}
It follows that 
$$
\norm{f}_{L^{\infty}(K)} \leq C_{n,p} \delta^{-n/p} \norm{f}_{L^p(\Omega)} + C_{n,p,\alpha} \delta^{\alpha} M.
$$
It is enough to choose 
$$
\delta = \left( \frac{\norm{f}_{L^p(\Omega)}}{M} \right)^{\frac{p}{n+\alpha p}}.
$$
\end{proof}

The next lemma shows that $\mA$-harmonic functions belong to the Sobolev space $W^{2,2}_{loc}$ whenever their gradient is nonvanishing. The argument is based on difference quotients and is well known (see e.g.~\cite[Theorem 2.103]{Maly}). However, since we could not find a reference for the precise result needed in this paper, the details are included below. 
See also~\cite[Theorem 2.5]{Pucci}, where it is proven that $\mA$-harmonic functions for $p\leq 2$ belong to $W^{2,p}_{loc}$ even without the additional assumption that the gradient is nonvanishing.

\begin{lemma}\label{sobo22}
Let $u\in W^{1,p}_{loc}(\Omega)$ be an $\mA$-harmonic function, where $\mA$ satisfies the assumptions of Proposition~\ref{Aharmonic_regularity}. Assume also that $\nabla u \neq 0$ in $\Omega$. Then $u\in W^{2,2}_{loc}(\Omega)$.
\end{lemma}
\begin{proof}
First recall that $u \in C^{1+\alpha}(\Omega)$ for some $\alpha > 0$~\cite[Thm. 2]{DiBenedetto}, so the condition $\nabla u \neq 0$ may be understood pointwise. We adopt the notation 
$$
\dh f(x) = f(x+h) - f(x).
$$
A function $f$ belongs to $W^{1,p}_{loc}(\Omega)$ if and only if $f \in L^p_{loc}(\Omega)$ and 
\begin{equation*}
 ||\dh f||_{L^p(U)}\leq C|h|
\end{equation*}
for any $U \subset \subset \Omega$ and for sufficiently small $h\in \R^n$~\cite[Theorem 5.8.3]{Evans}. Thus, to prove the claim, we need to show that
\begin{equation*}
||\dh \nabla u||_{L^{2}(U)}\leq C |h|,
\end{equation*}
for any $U\subset \subset \Omega$ and for any $|h|$ sufficiently small.

Since $u$ is $\mA$-harmonic, we have for any test function $\varphi \in W^{1,2}(\Omega)$ that is compactly supported in $\Omega$ and for any $h$ sufficiently small that
\begin{equation}\label{weak_difference}
 \int_\Omega \left[ \mA(x+h,\nabla u(x+h)) - \mA(x,\nabla u(x)) \right] \cdot \nabla \varphi(x) \,dx = 0.
\end{equation}
Choose $V$ so that $U\subset \subset V \subset \subset \Omega$. We will use the test function 
$$
\vp=\eta^2 \dh u\,,
$$
where $\eta \in C^{\infty}_c(V)$, $0 \leq \eta \leq 1$, is a cutoff function satisfying $\eta = 1 \text{ near } \overline{U}$.

We separate the expression in the integrand of \eqref{weak_difference} above into two parts:
\begin{align}
  &\mA(x+h,\nabla u(x+h)) - \mA(x,\nabla u(x)) = I_1(x) + I_2(x), \label{I_separation}\\
  &\quad I_1(x) = \mA(x,\nabla u(x+h))-\mA(x,\nabla u(x)), \nonumber \\
  &\quad I_2(x) =  \mA(x+h,\nabla u(x+h))-\mA(x,\nabla u(x+h)). \nonumber
\end{align}
For the first term we write 
\begin{align*}
 \mA^k&(x,\nabla u(x+h))-\mA^k(x,\nabla u(x))\\
  & \qquad =\int_0^1\frac{\partial}{\partial t} \mA^k(x,t \nabla u(x+h)+(1-t) \nabla u(x)) \,dt \\
  & \qquad =\int_0^1 \partial_{\xi_j} \mA^k(x,\dh^t\nabla u(x)) \dh \p_j u(x) \,dt.
\end{align*}
Here we have denoted
\begin{gather*}
\dh^t \nabla u(x)=t \nabla u(x+h)+(1-t) \nabla u(x).
\end{gather*}
Let us denote by $D\mA(\xi)$ the matrix $(\p_{\xi_j}\mA^k(x,\xi))$, $j,k=1,\ldots,n$, where we have suppressed the $x$ variable from our notation temporarily. Taking the inner product with the gradient of the test function $\varphi=\eta^2 \dh u$ yields
\begin{align*}
  &(D\mA(\dh^t\nabla u) \dh \nabla u) \cdot \nabla \varphi  \\
  &= \eta^2 (D\mA(\dh^t\nabla u)\, \dh \nabla u) \cdot \nabla (\dh u) + 2\eta\dh u  (D\mA(\dh^t\nabla u) \dh \nabla u) \cdot \nabla \eta  \\
  &=S_1+S_2.
\end{align*}
Now, since $\dh \nabla u=\nabla \dh u$, the $S_1$ term satisfies
$$
S_1 \geq c_1 \eta^2|\dh^t\nabla u|^{p-2} |\dh \nabla u|^2.
$$
Here we have used the ellipticity condition~\eqref{Aharmonic_reg_assumptions}. To estimate the $S_2$ term, we use Young's inequality with $\e > 0$, 
\begin{align*}
|S_2|&\leq 2 |D\mA(\dh^t\nabla u)| |\eta\, \dh \nabla u| |\dh u\,  \nabla \eta| \\
  & \leq  |D\mA(\dh^t\nabla u)|\l(\frac{1}{\e}|\dh u\,  \nabla \eta|^2+ \e |\eta\, \dh  \nabla u|^2\r).
\end{align*}

Let us next estimate the inner product of $\nabla \varphi$ and the $I_2$ term in~\eqref{I_separation}. The assumption~\eqref{Aharmregass} and the mean value theorem 
yield
\begin{align*}
  | I_2(x)& \cdot \nabla \varphi(x) | \leq|\mA(x+h,\nabla u(x+h))-\mA(x,\nabla u(x+h))||\nabla \varphi(x)| \\
  & \leq c_2 |h| |\nabla u(x+h)|^{p-1} |\nabla \varphi(x)| \\
  & \leq c_2 |h| |\nabla u(x+h)|^{p-1} \l(|\eta(x) \dh \nabla u(x)| + 2 |\eta \dh u(x) \nabla \eta(x)|\r).
\end{align*}

Until now, we have only used the structural assumptions of $\mA$. We apply next the additional conditions $u \in C^{1+\alpha}$ and $\nabla u \neq 0$ in $\Omega$, so $\abs{\nabla u} \geq \delta > 0$ on $V$. Using these facts, we first see that
\begin{align*}
S_1&\geq c_3 \eta^2 |\dh \nabla u|^2 \\
  |S_2| &\leq  c_4 \l(\frac{1}{\e}|\dh u\,  \nabla \eta|^2+ \e |\eta\, \dh  \nabla u|^2\r) \\
| I_2 \cdot \nabla \varphi | & \leq c_5 |h| \l(|\eta \dh \nabla u| + |\dh u\,\nabla \eta|\r).
\end{align*}
Integrating and using the Cauchy-Schwarz inequality then yields
\begin{align*}
0&= \int_\Omega \left[ I_1+I_2 \right] \cdot \nabla \varphi \,dx \\
  &\geq C_1 ||\eta \dh \nabla u||_{L^2(V)}^2- C_2 \e ||\eta \dh \nabla u||_{L^2(V)}^2 - C_3 |h|\, ||\eta \dh \nabla u||_{L^2(V)} \\
  & - \frac{C_4}{\e} ||\dh u||_{L^2(V)}^2 - C_5 |h|  ||\dh u||_{L^2(V)}.
\end{align*}
Choosing $\e$ and $h$ so small that $C_1-C_2\e - C_3 \abs{h} \geq C_6>0$ and using the mentioned fact about Sobolev spaces that 
$$
||\dh u||_{L^2(V)}\leq C_7|h|,
$$
for all $h$ small enough, yields
$$
||\dh \nabla u||_{L^2(U)}^2 \leq C_8 |h|^2.
$$
By the discussion in the beginning of the proof, this shows the claim.
\end{proof}

\section*{Acknowledgements}

T.L.~is supported by the Finnish National Graduate School in Mathematics and its Applications, and M.S.~is partly supported by the Academy of Finland and an ERC Starting Grant. The authors would like to thank Juha Kinnunen and Xiao Zhong for helpful discussions on $\mA$-harmonic functions, Luca Capogna for explaining the work of Lelong-Ferrand, and Marc Troyanov for pointing out the references \cite{Reshetnyak78} and \cite{Shefel}.

\end{document}